\documentclass[11pt,letterpaper]{amsart}

\usepackage{latexsym,color,amsmath,amsthm,amssymb,amscd,amsfonts}
\usepackage{tikz}
\usetikzlibrary{arrows}
\usepackage{scalefnt}
\usepackage[font=small,labelfont=bf]{caption}
\usepackage{float}
\usepackage{graphicx}
\usepackage{relsize}
\usepackage{amsopn}
\usepackage{mathrsfs}  %
\usepackage[hidelinks]{hyperref}
\usepackage{bm}
\usepackage{bbm}
\usepackage[shortlabels]{enumitem}
\usepackage{multicol}
\usepackage{float} 
\usepackage{todonotes}
\usepackage{cite}
\usepackage{soul}
\usepackage[margin=2.5cm]{geometry}

\usepackage{amsrefs}

\graphicspath{ {images/} }
\newtheoremstyle{nonum}{}{}{\itshape}{}{\bfseries}{.}{ }{\thmnote{#3}}

\newtheorem{thm}{Theorem}[section]
\newtheorem*{thm*}{Theorem}

\newtheorem{cor}[thm]{Corollary}
\newtheorem{lem}[thm]{Lemma}
\newtheorem{prop}[thm]{Proposition}

\newtheorem{definition}[thm]{Definition}
\newtheorem*{definition*}{Definition}

\newtheorem{manualtheoreminner}{Theorem}
\newenvironment{manualtheorem}[1]{%
	\renewcommand\themanualtheoreminner{#1}%
	\manualtheoreminner
}{\endmanualtheoreminner}

\newtheorem{manualdefinitioninner}{Definition}
\newenvironment{manualdefinition}[1]{%
	\renewcommand\themanualdefinitioninner{#1}%
	\manualdefinitioninner
}{\endmanualtheoreminner}

\newcommand{\innerthmname}{}%

\theoremstyle{definition}

\theoremstyle{nonum}

\theoremstyle{remark}

\newtheorem*{rems*}{Remarks}

\newtheorem{rem}[thm]{Remark}

\newcommand{\R}{\mathbb R}
\newcommand{\RR}{\mathbb R}

\newcommand{\N}{\mathbb N}

\def\L{{\mathcal L}}

\newcommand{\iprod}[2]{\langle #1,#2 \rangle} %

\def\mathcalL{{\mathcal L}}
\def\L{{\mathcalL}}

\def\grad{{\nabla}}

\def\Id{{\rm{\it Id}}}

\def\eps{{\varepsilon}}

\def\L{{\mathcal L}}

\def\mathcalL{{\mathcal L}}
\def\L{{\mathcalL}}

\def\grad{{\nabla}}

\def\Id{{\rm{\it Id}}}

\def\eps{{\varepsilon}}

\begin{document}
\title {A Rockafellar-type theorem for non-traditional costs}
\date{}
\author{S. Artstein-Avidan, S. Sadovsky, K. Wyczesany}

\begin{abstract}
In this note, we present a unified approach to the problem of existence of a potential for the optimal transport problem with respect to non-traditional cost functions, that is, costs that assume infinite values.  We establish a new method that relies on proving solvability of a special (possibly infinite) family of linear inequalities. When the index set of this family is countable, we give a necessary and sufficient condition on the coefficients that assures the existence of a solution, and which, in the setting of transport theory, we call $c$-path-boundedness. In the case of an uncountable index set, one needs an additional assumption for solvability$^{1}$\thanks{$^{1}$This is a revised version that corrects a mistake in the original paper regarding the solvability of an uncountable family of inequalities.}. 
We propose a sufficient condition in this case.  We note that any set admitting a potential must be $c$-path-bounded, and this condition replaces $c$-cyclic monotonicity from the classical theory,  i.e.  when the cost is real-valued. Our method also gives a new and elementary proof for the classical results of Rockafellar, Rochet and R\"uschendorf.

\end{abstract}

\maketitle

\section{Introduction} 
Mass transport problems have been widely studied in mathematics and have a vast array of applications and implications, for an overview, see \cites{villani-book,ambrosio-gigli}. The classical mass transport problem of Monge is that of finding, given probability measures $\mu$ and $\nu$, the infimal total cost of a mapping $T$ mapping one to the other, that is,   
\[\inf \{ \int c(x,Tx)d\mu: T_{\#}(\mu)=\nu\}.\]
Here, $c$ is a cost function, which one may or may not assume has some good continuity or smoothness results (in case the integral is not defined, one may use an outer integral). 

The celebrated theorem of Brenier \cite{brenier} states that for the quadratic cost, $c(x,y)=\|x-y\|^2$ one may find $T$ which attains the infimum, and moreover, this optimal map (which is called the Brenier map) is given as the gradient of some convex function $\varphi$. The function $\varphi$ for which $\grad \varphi = T$ is called a potential for the map $T$. This elegant result, which has many applications, is proven using an important geometric interpretation of the notion of optimality, named \emph{cyclic monotonicity} of the support of the optimal plan. Rockafellar's 
Theorem \cite{rockafellar} states that a set is cyclically monotone if and only if it lies in the subgradient of some convex function.  The generalization to other finite-valued costs was established by Rochet and R\"uschendorf \cites{rochet, ruschendorf}.  The notion of $c$-cyclic monotonicity (called simply cycylic monotonicity when the cost is quadratic) is defined as follows : 
We say that a set $G\subset X\times Y$ is $c$-cyclically monotone if for any finite set $\{(x_i,y_i)\}_{i=1}^m\subset G$, 
\[\sum_{i=1}^m c(x_i,y_i)\le \sum_{i=1}^m c(x_{i+1},y_{i}),\]
where $x_{m+1} = x_1$. The Rochet-R\"uschendorf theorem 
states that a set is $c$-cyclically monotone if and only if it lies in the $c$-subgradient of a $c$-class function $\varphi$ (see definitions in Sections \ref{sec:csg} and \ref{sec:ccm}). In such a case we say that $\varphi$ is a $c$-potential for the set.

This condition of $c$-cyclic monotonicity is very useful, and when the cost have only finite values then a plan which is concentrated on a set $G$ which is $c$-cyclically monotone must be optimal, and admit a $c$-potential. However, this statement fails when the cost is allowed to attain the value $+\infty$, namely for \emph{non-traditional costs}. In such cases optimality of a plan is not equivalent to $c$-cyclic monotonicity of the support (see Example 3.1 in \cite{ambrosio-pratelli}).

In this work, we modify the condition of $c$-cyclic monotonicity in such a way so that we may extend the results of  Rockafellar, Rochet and R\"uschendorf to the non-traditional setting.

\begin{definition}%
	Fix sets $X,\,Y$ and $c:X\times Y\to (-\infty, \infty]$. A subset $G\subset X\times Y$ will be called {\bf $c$-path-bounded} if $c(x,y)<\infty$ for any $(x,y)\in G$, and for any $(x,y)\in G$ and $(z,w)\in G$, there exists a constant $M=M((x,y), (z,w))\in \RR$ such that the following holds: For any $m\in {\mathbb N}$ and any $\{(x_i, y_i): 2\le i\le m-1\} \subset G$, denoting $(x_1, y_1) = (x,y)$ and $(x_m, y_m) = (z,w)$,   we have 
	\[ \sum_{i=1}^{m-1} \big(c(x_i, y_i)-c(x_{i+1}, y_{i})\big) \le M.  \]   
\end{definition}

It is not hard to check that the above is a {\em stronger} condition than $c$-cyclic monotonicity (see Section \ref{sec:ARRresult}), and equivalent to it when the cost is finitely valued. For a given plan to have a $c$-potential   
it must be concentrated on a
$c$-path-bounded set. We address the reverse implication in this paper. The only additional condition needed is one that guarantees $G$ is not, in some sense, completely split into two components.

  \begin{definition}\label{def:black-hole-cost}
 	Let $c:X\times Y \to(-\infty , \infty]$ be a cost function. We say that a set $G\subset X\times Y$ does not have an infinite black hole if for every   infinite subset $G_0\subset G$ there exists $y\in P_YG_0$ and $z\in P_X(G\setminus G_0)$ such that $c(z,y)< \infty$.
 \end{definition}
 
We stress that no measurability assumptions are needed, and the result is purely a set-theoretic one, stripping measurability off  of this part of measure-transport theory. 
Our main theorem is: 
\begin{manualtheorem}{1.1}\label{thm:Non-tradRR}
Let $X,\,Y$ be two arbitrary sets and $c:X\times Y\to (-\infty, \infty]$ an arbitrary cost function. Assume that $G\subset X\times Y$ is a $c$-path-bounded subset that is countable, or if it is uncountable, it does not have an infinite black hole. Then there exists a $c$-class function $\varphi:X\to [-\infty, \infty]$ such that $G\subset \partial^c \varphi$.
\end{manualtheorem}

Our main tool is a new, and seemingly unnoticed, fact about the solvability of an infinite family of linear inequalities. We mention that using it for the case where $\alpha_{i,j}$ are all finite, we get a new elementary proof for the classical Rockafellar-Rochet-R\"uschendorf theorem. 

\begin{manualtheorem}{1.2}
	Let $\{\alpha_{i,j}\}_{i,j \in I}\in [-\infty, \infty)$, where $I$ is some arbitrary index set, and with $\alpha_{i,i} = 0$. The system of inequalities 
	\begin{equation}\label{eq:linear-sys} \alpha_{i,j} \le x_i - x_j, \quad i,j \in I \end{equation} 
	has a solution if 
    (a) for any $i,j\in I$ there exists some constant $M(i,j)$ such that for any $m$ and any $i_2,\cdots,i_{m-1}$, letting $i = i_1$ and $j = i_m$ one has that  $\sum_{k=1}^{m-1} \alpha_{i_k, i_{k+1}}\le M(i,j)$, and \\
    (b) either $I$ is at most countable, or, if $I$ is uncountable then for every {infinite subset} $J\subset I$ there exist some $j\in J, i\notin J$ with $\alpha_{j,i}>-\infty$.
\end{manualtheorem}
The proof to the theorem above appears in  
Section \ref{sec:ARRresult}. 

While in this work we focus on the characterization of sets $G\subset X \times Y$ that admit a $c$-potential with respect to an arbitrary cost without mentioning measures, one should keep in mind that in applications $G$ will be a set on which a transference plan between two probability measures is concentrated. Such applications were studied by Ambrosio and Pratelli \cite{ambrosio-pratelli},   Pratelli \cite{pratelli},  Bianchini and Caravenna \cite{bianchini-caravenna,bianchini-caravenna-long} and Beiglb\"ock, Goldstern, Maresch, Schachermayer \cite{optimal-and-better}.  It was shown in the latter that the existence of a $c$-potential for a plan with finite total cost implies optimality. 

This work is also motivated by the case of the so-called {\em polar cost},   connected to the polarity transform, which has received much attention lately \cite{hidden,shiri-yanir,artstein-florentin-milman}.   This is an important example to which the Rockafellar-Rochet-R\"uschendorf theorem %
does not apply. Indeed, when $c$ is the polar cost, one may find $c$-cyclically monotone sets which are not contained in the $c$-subgradient of any $c$-class function (see Section \ref{sec:example-non}). Transportation with respect to polar cost was first considered in \cite{hila}, and the results obtained there are to appear in \cite{paper3}.

Special families of non-traditional costs were also studied, among others, by Bertrand and Puel \cite{bertrand-puel}, and  Bertrand, Pratelli and Puel \cite{bertrand-pratelli-puel}, who considered relativistic cost functions of the form $c(x,y)=h(x-y)$, where $h$ is a strictly convex and differentiable function, restricted to a strictly convex set $K$, and infinite outside $K$, see also \cite{jimenez-chloe-santambrigio}. A special case of such a cost is the relativistic heat cost, linked with the relativistic heat equation considered by Brenier \cite{brenier-relativ-cost}. Further, a non-traditional cost function on the sphere is used in \cite{bertrand-gauss-curvature} for a proof of Alexandrov's theorem about prescribing the Gauss curvature of convex sets in the Euclidean space, following Oliker \cite{oliker}. Further results about existence of transport plans which admit a $c$-potential with respect to non-traditional costs can be found in \cite{paper2}. 

{\textbf{Structure of the paper}}: In Section \ref{sec:bk} we provide the background for the problem, and define the notions of $c$-transform, $c$-subgradient, and $c$-cyclic monotonicity. We connect these notions to the classical example of the Legendre transform and the more recent polarity transform. We then explain the Rockafellar-Rochet-R\"uschendorf theorem and show that it does not hold for non-traditional costs.  In Section \ref{sec:ARRresult} we define $c$-path-bounded sets, prove our main theorem, and show that it implies the  Rockafellar-Rochet-R\"uschendorf theorem for traditional costs. In Section \ref{sec:sc} we show some consequences of our theorem under further restrictions, in particular that it implies many of the previous results. 

\textbf{Acknowledgement}: The first named author would like to thank Z.~Artstein for useful conversations. The authors received funding from the
European Research Council (ERC) under the European Union’s Horizon 2020
research and innovation programme (grant agreement No 770127). The second named author is grateful to the Azrieli foundation for the award of an Azrieli fellowship. We sincerely thank Yuan Gao (University of British Columbia) for bringing a mistake in the original version of the paper to our attention.

\section{Costs and potentials}\label{sec:bk}

\subsection{Mass transport} 

The initial data of a mass transport problem is $(X,\Sigma_X),(Y, \Sigma_Y)$, 
where $X,\,Y$ are sets and $\Sigma_X, \Sigma_Y$ are $\sigma$-algebras, and a cost function $c:X\times Y\to (-\infty, \infty]$, 
which is assumed to be measurable (with respect to the product $\sigma$-algebra). One usually considers $X$ and $Y$ which are separable metric spaces endowed with the Borel $\sigma$-algebra, and in particular $X = Y =  \RR^n$  serve as a good model space.

Given  $\mu\in {\mathcal P(X)}$ a probability measure on $X$, and $\nu\in {\mathcal P}(Y)$, we say that a measurable map $T:X\to Y$ is a transport   map  between $\mu$ and $\nu$ if 
\[ \mu(T^{-1}(B)) = \nu(B)\]
for all measurable $B\in \Sigma_Y$.

We say $\gamma \in {\mathcal P}(X\times Y)$ is a transport plan between $\mu$ and $\nu$, and denote $\gamma\in \Pi(\mu, \nu)$, if its  marginals on $X$ and $Y$ are $\mu$ and $\nu$ respectively. In particular, to any transport map $T$ there corresponds (in the obvious way) a transport plan concentrated on its graph, denoted usually by $(Id,T)_{\#}\mu$.

The total cost of transporting $\mu$ to $\nu$ is defined by 
\[ C(\mu, \nu) = \inf\left\{ \int_{X\times Y} c(x,y) d\gamma: \gamma \in \Pi(\mu, \nu)\right\}.\] 

A central theorem of Kantorovich \cite{kantorovich,kantorovich2} states that in the case of a lower semi-continuous cost function (for further generalisations see e.g. \cite{general-duality-MK}), the above infimum is attained and is given by 
\begin{equation}\label{eq:thm-kantorovich}
\begin{aligned}
C(\mu, \nu) = \sup\bigg\lbrace&\int_X\varphi d\mu + \int_Y\psi d\nu:\\ 
&\varphi \in L_1(X, d \mu ), \, \psi\in L_1(Y, d\nu), \, (\varphi, \psi) {\rm ~ admissible}\bigg\rbrace, 
\end{aligned}
\end{equation}
where a pair of functions is called admissible if $\varphi(x)+\psi(y) \le c(x,y)$ for all $x\in X,\, y\in Y$.

\subsection{The $c$-class and basic functions}

Given a cost $c:X\times Y \to (-\infty, \infty]$ (no measurability assumptions), and a function $\varphi:X\to [-\infty, \infty]$, we define its $c$-transform $\varphi^c:Y\to [-\infty, \infty]$ by 
\[ \varphi^c(y) = \inf_{x\in X} \left(c(x,y)-\varphi(x)\right).\]   
Similarly, when $\psi: Y\to [-\infty, \infty]$, we use the same notation to define
\[ \psi^c(x) = \inf_{y\in Y} \left(c(x,y)-\psi(y)\right).\]
In the case where the expression considered on the right hand side is $+\infty - (+\infty)$, we stipulate that this quantity is equal to $+\infty$. This corresponds to the fact that for $\varphi(x) + \psi(y) \le \infty$ to hold true, no condition on $\varphi(x)$ or $\psi(y)$ is needed. 
To avoid arithmetic manipulations with infinite numbers, one may instead consider the infimum in the definition of $\varphi^c$ to be taken only  over those {points} $x$ for which $c(x,y)<\infty$, and similarly for $\psi^c$.

It is easy to check that if $(\varphi, \psi)$ is an admissible pair then $\psi \le \varphi^c$ and $\varphi \le \psi^c$, namely the $c$-transform  
maps a function $\varphi$ to the (point-wise) largest $\psi$ such that the pair $(\varphi, \psi)$ is admissible.  In particular, 
 $\varphi^{cc} \ge \varphi$, since $(\varphi, \varphi^c)$ is an admissible pair. As a result of this observation we get that $\varphi^{ccc} = \varphi^c$, since in addition the mapping $\varphi\mapsto \varphi^c$ is order reversing. 
 
In light of this discussion, we define the $c$-class associated with a cost function $c:X\times Y\to (-\infty, \infty]$ as the image of the $c$-transform, namely 
 \[ \{\varphi^c: \, \varphi:X\to [-\infty, \infty]\}{\rm~~ or ~~} 
 \{\psi^c: \,\psi:Y\to [-\infty, \infty]\}. \] 
When $X = Y$ and $c(x,y) = c(y,x)$ the two transforms coincide, as do the two $c$-classes. We slightly abuse notation, since throughout this note we will be considering \textbf{symmetric} cost functions,  by referring to ``the $c$-class'' where in fact we should formally make a distinction   between the $c$-class of functions on $X$ and the $c$-class of functions on $Y$. 

Within the $c$-class, we define the sub-class of basic functions to be functions of the form $\varphi(x) = c(x,y_0) + \beta$ (and $\psi(y) = c(x_0, y)+\beta$, respectively) for some $x_0\in X, \, y_0\in Y$ and $\beta \in \RR$. 
By definition, every $c$-class function is an infimum of basic functions, and it is not hard to check that the $c$-class is   closed under infimum. 

\subsection{The $c$-subgradient}\label{sec:csg}

Given a cost $c:X\times Y \to (-\infty, \infty]$, and a $c$-class function $\varphi$, we define its $c$-subgradient by 
\[ \partial^c \varphi = \{ (x,y) \in X\times Y: \varphi(x) + \varphi^c(y) = c(x,y) <\infty \}.\]  
Clearly $(x,y) \in \partial^c \varphi$ if and only if $(y,x) \in \partial^c \varphi^c$.  

As a motivation for this definition one may go back to Kantorovich's theorem, recalled in \eqref{eq:thm-kantorovich}, and note that a one sided inequality is trivial, since for any $\gamma\in \Pi(\mu, \nu)$, and any admissible pair $\varphi, \psi$, we have
\[  \int_X\varphi d\mu + \int_Y\psi d\nu = \int_{X\times Y} (\varphi(x) + \psi(y)) d\gamma \le \int_{X\times Y} c(x,y) d\gamma, \]
and for equality to hold one needs that $\gamma$-almost everywhere in $X\times Y$, we will have $\varphi(x) + \psi(y) = c(x,y)$. Since replacing $\psi$ by $\varphi^c$ only increases the integral, we see that in fact for equality to hold we need the plan $\gamma$ to be concentrated on the $c$-subgradient of some $c$-class function $\varphi$. We stress that this only serves as motivation and will not be used in this paper, in particular we have not assumed any measurability for the cost or the functions, and the functions we later consider might not be measurable or integrable.

\subsection{$c$-cyclic monotonicity}\label{sec:ccm}

Given a cost $c:X\times Y \to (-\infty, \infty]$, a subset $G\subset X\times Y$ is called $c$-cyclically monotone if $c(x,y)<\infty$ for all $(x,y)\in G$, and for any $m$, any $\{(x_i, y_i): 1\le i\le m\} \subset G$, and any permutation $\sigma$ of $[m] = \{1, \ldots,m\}$ it holds that 
\begin{equation}\label{eq:ccm}  \sum_{i=1}^m c(x_i, y_i) \le \sum_{i=1}^m c(x_i, y_{\sigma(i)}).  \end{equation}
 To our best knowledge, this definition was first introduced by Knott and Smith \cite{knott-smith}, as a generalization of cyclic monotonicity considered by Rockafellar \cite{rockafellar} in the case of quadratic cost. When $G\subset \partial^c\varphi$, we have by definition that 
\[
c(x_i, y_i) = \varphi(x_i) + \varphi^c(y_i) \quad {\rm and} \quad c(x_i, y_{\sigma(i)}) \ge \varphi(x_i) + \varphi^c(y_{\sigma(i)}),  \]
and summing these inequalities over $i\in [m]$ gives \eqref{eq:ccm}.
In other words,  $c$-cyclic monotonicity is a necessary condition for a set $G\subset X\times Y$ to have a {\bf{$c$-potential}} $\varphi$, that is, a $c$-class function for which $G\subset \partial^c\varphi$.

\subsection{Example: The  {quadratic} cost}

The most well studied example is that of quadratic cost, namely $X = Y = \RR^n$ and the Euclidean distance squared $\|x-y\|^2/2$. From the transportation perspective, since $\|x-y\|^2/2 = \|x\|^2/2 -\iprod{x}{y} + \|y\|^2/2$, only the mixed term is important. We will thus work with $c(x,y) = -\iprod{x}{y}$, bearing in mind that in this setting the cost is no longer positive, and is unbounded  both from above and below. It  is, however, traditional, namely does not assume the value $+\infty$. 

 The $c$-transform is given by 
\[ \varphi^c(y) = \inf_{x\in \RR^n}\left( -\iprod{x}{y} - \varphi(x)\right), \] 
which can be written as
\[ -\varphi^c = {\mathcal L} (-\varphi),\]
where $\L \psi(y) = \sup_{x}(\iprod{x}{y}-\psi(x))$ is the Legendre transform  (for an overview see e.g. \cite{rockafellar-book}).
It is easy to check that the associated $c$-class is that of all concave functions on $\RR^n$ which are upper semi-continuous.

\subsection{The Rockafellar-Rochet-R\"uschendorf Theorem}\label{subsec: RRR-thm}

For costs $c:X\times Y\to \RR$, namely traditional costs, this theorem states that the condition of $c$-cyclic monotonicity for a set $G\subset X\times Y$ is equivalent to the condition that there exists a $c$-class function such that $G\subset \partial^c \varphi$ \cite{rockafellar, ruschendorf}. 

This fundamental theorem has a beautiful constructive proof. Indeed, one fixes some element $(x_0, y_0)\in G$ and  sets 
\begin{align}\label{def:Rock-function}
    \varphi(x) = \inf_{x,\, (x_i, y_i)_{i=1}^m \subset G}
\left( c(x,y_m) - c(x_0, y_0) + \sum_{i=1}^m\left( c(x_i, y_{i-1}) - c(x_i, y_i)\right) \right).
\end{align} 
 
    Interestingly, examining the proof (which can be found in many places, see for example \cite{villani-book}), there is only one step where the $c$-cyclic monotonicity is used, and it is to prove that at $(x_0, y_0)$ the function $\varphi$ defined above  is finite. For non-traditional costs the proof fails (as does the theorem, as we shall shortly see), and the function $\varphi$ defined above may be infinite on points $x$ with $(x,y)\in G$, in which case $(x,y)$ cannot belong to $\partial^c \varphi$.

Below we present a corresponding theorem for non-traditional costs. Applied to the case of traditional costs, our method gives a new proof for the Rockafellar-Rochet-R\"uschendorf theorem, which is, in our view,  more intuitive.

\subsection{Known results}\label{subsec:known-results}

Let us emphasize that in this note we are solely devoted to investigating when, for a given set, one can find a $c$-potential, and therefore there are no measures involved. However, the literature on the optimality of transport plans mentioned in the introduction %
includes relevant ideas and results on the topic at hand. 

In particular, an equivalence relation on the elements of a set $G\subset X\times Y$ was introduced \cite[Chapter 5, p.75]{villani-book} and studied \cite{bianchini-caravenna,optimal-and-better}. It is defined as follows: we may associate with $G$ a directed graph with a vertex set $G$ in which there is an edge from $(x,y)$ to $(z,w)$ if $c(z,y)<\infty$. 
On the vertex set of this graph (namely on $G$) we define an equivalence relation $\sim$, where two points are equivalent, $(x,y) \sim (z,w)$, if there exists a directed cycle passing through both (or, equivalently, if there is a directed path from each of the points to the other). This is clearly an equivalence relation.

The following proposition, which will straightforwardly follow from our analysis (see Section \ref{sec:sc}), appeared in \cite[Proposition 3.2.]{optimal-and-better} in a slightly different form.  Its formulation in \cite{optimal-and-better} includes a transport plan between two Borel probability measures,  and the assumption that the cost is bounded from below.  It was proved using the Rockafellar-type function \eqref{def:Rock-function} whereas our proof is different and elementary.

\begin{prop}\label{prop:1equivClassIMPLIEScpathbdd}
Let $c: X\times Y \to   (-\infty ,\infty]$ be a cost function and let $G\subset X \times Y$ be a $c$-cyclically monotone set. Assume the equivalence relation $\sim$ defined above had just one equivalence class. Then there exists a $c$-potential for $G$, i.e. a function $\varphi$ such that $G\subset \partial ^c \varphi$. 
\end{prop}

It was further established in \cite{optimal-and-better},  that the existence of a $c$-potential for a finite transport plan between two measures is equivalent to the plan being \emph{robustly optimal} (see Definition 1.6. in \cite{optimal-and-better}).

\subsection{An example of a  $c$-cyclically monotone set with no $c$-potential}\label{sec:example-non} 

To end this section, and before moving to the proof of our main result, let us present an example of a $c$-cyclically monotone 
set which is not a subset of $\partial^c\varphi$ for any $c$-class $\varphi$. In other words, $G$ has no potential (no offense, $G$). One could extract an example from \cite{ambrosio-pratelli}, but we shall use the polar cost in one dimension, namely $c(x,y) = -\ln (xy-1)$ on $\RR\times \RR$ (where the logarithm of a non-positive number is defined to be $-\infty$), to underline its importance %
and draw attention to the fact that the regularity of the cost function does not play a role in the existence of a $c$-potential. 

For the polar cost,  $c$-cyclic  monotonicity is equivalent to the set lying on the graph of a decreasing function. More precisely we prove the following lemma. 

\begin{lem}\label{lem:pcm-oneD} Let $c(x,y) = -\ln (xy-1)$.  
	A set $G\subset \RR^+\times \RR^+$ is $c$-cyclically monotone 
	if and only if for any $(x_1,y_1), (x_2, y_2)\in G$ one has $(x_1 -x_2)(y_1 - y_2)\le 0$.  
\end{lem}

\begin{proof}
	In one direction, assume $G$ is $c$-cyclically monotone and let $(x_1,y_1)\in~G$, $(x_2, y_2)\in~G$. In particular $x_1y_1>1$ and $x_2y_2>1$.  {The}  $c$-cyclic  monotonicity    implies 
	\[ c(x_1, y_1) + c(x_2, y_2) \le c(x_1, y_2) + c(x_2, y_1),\] 
	which we rewrite as
	\[ \ln (x_1y_2-1)+ \ln (x_2y_1-1)\le \ln (x_1y_1-1)+ \ln (x_2y_2-1). \] 
	We may assume that $x_1y_2>1$ and $x_2y_1>1$, or else there is nothing to prove. We thus know that
	$(x_1y_2-1) (x_2y_1-1)\le (x_1y_1-1) (x_2y_2-1)$, 
	which after the rearrangement becomes  $(x_1-x_2)(y_1 - y_2) \le 0$,  as claimed. 
	
	For the other direction, given some $G$ which satisfies that for any $(x_1,y_1)\in G$ and $(x_2, y_2)\in~G$ one has $(x_1 -x_2)(y_1 - y_2)\le 0$, pick some $m$-tuple in $G$. 
	It is enough to show that the subset $\{(x_i, y_i): i\in [m]\}$ is $c$-cyclically monotone. However, as this is a finite subset, one may consider all the sums 
	$C(\sigma) = \sum_{i=1}^m c(x_i, y_{\sigma(i)})$ for some permutation $\sigma\in S_m$ of the set $[m]$. The permutation with minimal cost $C(\sigma)$ (not necessarily unique) will satisfy, by definition, that the set  $\{(x_i, y_{\sigma(i)}): i\in [m]\}$ is $c$-cyclically monotone. In particular, by the first direction, we will have 
	$(x_i -x_j)(y_{\sigma(i)} - y_{\sigma(j)})\le 0$. However, there is only one decreasing rearrangement of a set $\{y_i: i\in [m]\}$, and we know already that the identity permutation ${\Id}$ is such a rearrangement, so in particular $C({Id})$ is minimal. We conclude that $G$ is  	 $c$-cyclically monotone.   
\end{proof}

 Our  {example} set $G$ will be the following 
\[ G = \{ (x,y): \tfrac34\le x<1,\, y = 3-2x \}\cup \{(\tfrac32, \tfrac34)\}.\] 
It is easy to check that $G$ satisfies the condition in Lemma \ref{lem:pcm-oneD}, namely it is  $c$-cyclically monotone for the cost $c(x,y) = -\ln (xy-1)$. 
However,  for the points $(x, y) = (\frac34, \frac32)$ and $(z, w) = 
(\frac32, \frac34)$, we may always add a third point $(x_2, y_2)\in G$ such that 
the following expression 
\[ c(x , y ) - c(x_2, y ) + c(x_2, y_2) - c(z, y_2) \] 
is arbitrarily large.

Indeed, picking $(x_2, y_2) = (t,3-2t)$,   compute 
\begin{eqnarray*}   
	&&c(x , y ) - c(x_2, y ) + c(x_2, y_2) - c(z, y_2) \\&&=
	\ln  8 +  \ln (\frac{3t}{2} -1)    - \ln ((t-1)(1-2t))  
	+\ln (\frac{7}{2} -t) \\  &&\ge  \ln8 -\ln 2 + \ln \frac52 - \ln ((t-1)(1-2t)). 
\end{eqnarray*}
As $t\to 1^{-}$, we have that $(t-1)(1-2t) \to 0^+$. In particular the term depending on $t$ in the lower bound for the expression tends to $+\infty$, and there can be no upper bound for it which does not depend on $t$, but only on the endpoints $ (\frac34, \frac32)$ and $ 
(\frac32, \frac34)$. 

This excludes the possibility for the existence of a $c$-potential, since if there existed some $\varphi$ such that 
$G\subset \partial^c \varphi$, it would necessarily satisfy (as $(x,y) \in G$ and $(x_2, y_2) \in G$) that
\[ c(x , y ) - c(x_2, y ) \le \varphi(x) - \varphi(x_2)  \,\,{\rm and}\,\,c(x_2, y_2) - c(z, y_2)  \le \varphi(x_2) - \varphi(z),\] 
so that in particular 
\[ c(x , y ) - c(x_2, y ) + c(x_2, y_2) - c(z, y_2)  \le \varphi(x) -\varphi(z), \] 
getting a bound for the aforementioned expression which does not depend on $x_2$. 
We have thus shown that $G$ is a $c$-cyclically monotone set which  does not admit a $c$-potential.

\section{A Rockafellar-type result for non-traditional costs}\label{sec:ARRresult}

As we have demonstrated in the previous section, for non-traditional costs it may happen that a set is $c$-cyclically monotone but fails to have a  {$c$-potential}   $\varphi$ in the $c$-class, even when $c$ is continuous and the spaces considered are simply $\RR^n$. 
  The example in Section \ref{sec:example-non} gives a clue as to what condition to consider. We call it $c$-path-boundedness. 

\begin{manualdefinition}{1.1}\label{def:c-path-bdd}
Fix sets $X,\,Y$ and $c:X\times Y\to (-\infty, \infty]$. A subset $G\subset X\times Y$ will be called {\bf $c$-path-bounded} if $c(x,y)<\infty$ for any $(x,y)\in G$, and for any $(x,y)\in G$ and $(z,w)\in G$, there exists a constant $M=M((x,y), (z,w))\in \RR$ such that the following holds: For any $m\in {\mathbb N}$ and any $\{(x_i, y_i): 2\le i\le m-1\} \subset G$, denoting $(x_1, y_1) = (x,y)$ and $(x_m, y_m) = (z,w)$,   we have 
	\[ \sum_{i=1}^{m-1} \big(c(x_i, y_i)-c(x_{i+1}, y_{i})\big) \le M.  \]   
\end{manualdefinition}

It is not hard to see that a $c$-path-bounded set must be $c$-cyclically monotone (indeed, if $(x,y) = (z,w)$  then if there is some path for which the sum is positive, one can duplicate it many times to get paths with arbitrarily large sums). It is also not hard to check (using the same reasoning we used in the example above) that $c$-path-boundedness is a necessary condition for the existence of a $c$-potential (we do this in Section \ref{subsec:summanry}). Our main result (Theorem \ref{thm:Non-tradRR}) affirms that the condition of $c$-path-boundedness is equivalent to the existence of a $c$-potential, as long as the set $G$ is countable or does not have any infinite ``black holes'' (see Definition \ref{def:black-hole}). 


 %

\subsection{Reformulation of the problem}

One can reformulate the problem of finding a $c$-potential for a given set $G\subset X\times Y$ as a question regarding the existence of a solution to a linear system of inequalities (possibly infinitely many of them).

\begin{thm}\label{thm:potential-means-inequalities}
	Let $c: X \times Y \to   (-\infty,\infty]$  be a cost function and let $G\subset X\times Y$. 
	Then
	there exists a $c$-potential for $G$, namely a $c$-class function $\varphi: X\to [-\infty, \infty]$   such that $G \subset  \partial^c\varphi $  if and only if the following system of  {inequalities}
	\begin{equation}\label{eq:system-of-eq-for-potential}
	c(x,y) - c(z,y) \le \varphi(x) - \varphi(z), \end{equation}
	indexed by $(x,y), (z,w)\in G$, has a solution $\varphi: P_X G \to \RR$,  
	where $~{P_XG = \{ x \in X: \exists y   \in Y , (x,y)\in G\}}$.   
\end{thm}

\begin{proof}
	Assume that there exists a 	$c$-potential  $\varphi: X\to [-\infty, \infty]$ such that $G \subset \partial^c\varphi$. We may restrict $\varphi$ to $P_XG$, on which it must attain only finite values, because  $G \subset \partial^c\varphi$ means in particular that $\varphi(x) + \varphi^c(y) = c(x,y) <\infty$. For every $z\in P_X G$ we have 
	\[ \varphi(z) = \inf_{w\in Y} \left(c(z,w) - \varphi^c(w)\right) \le c(z,y) -\varphi^c(y). \]
	At the same time, since $(x,y) \in \partial^c \varphi$,  
	\[ \varphi(x) = \inf_{w\in Y} \left(c(x,w) - \varphi^c(w) \right)= c(x,y) -\varphi^c(y).  \]
	Taking the difference of the two equations we get 
	\[ c(x,y) - c(z,y) \le \varphi(x) - \varphi(z). \]  
	
	For the other direction, assume we have a solution to the system of  {inequalities}. We would like to extend it to some $c$-class function defined on $X$. To this end let
	\[ \tilde{\varphi}  (z)  = \inf_{(x,y)\in G}  \{ c(z, y ) - c(x,y) + \varphi(x)\}.\]
	We need to show that the function $\tilde{\varphi}$, which is clearly in the $c$-class, satisfies that it is an extension of the original function $\varphi: P_XG \to \RR$, and that it is a $c$-potential, namely $G\subset \partial^c\tilde{\varphi}$. 
	
	The assumption \eqref{eq:system-of-eq-for-potential} implies that for $z\in P_XG$ we have 
	\[ \varphi(z) \le c(z,y) - c(x,y) +  \varphi(x) \]  
	and so the infimum  in the definition of the extended $\tilde{\varphi}$ is attained at $z$ itself. In particular, we get that $\tilde{\varphi}$ is indeed an extension of the original function $\varphi$. This means that 
 if $(x,y)\in G$ then 
 \begin{eqnarray*} \tilde{\varphi}^c(y) &=&
 	 \inf_{z\in X}\left( c(z,y) - \tilde{\varphi}(z)\right)\\ & = & 
 	 \inf_{z\in X}\sup_{(x',y')\in G}\left( c(z,y) -
 	 c(z, y' ) + c(x',y') - \varphi(x')
 	 \right)\\& \ge & 
 	 \inf_{z\in X}\left( c(z,y) -
 	 c(z, y ) + c(x,y) - \varphi(x)\right)\\& =& c(x,y)-\varphi(x) = c(x,y) -\tilde{\varphi}(x).\end{eqnarray*}
 As the opposite inequality is trivial, we get 
 $(x,y)\in \partial^c \tilde{\varphi}$, as required. 
\end{proof}

The above theorem, while very simple in nature, reduces the question of finding a $c$-potential to the question of determining when a set of linear inequalities has a solution. The index set for the inequalities are pairs $((x,y),z)\in G\times P_XG$ (or, equivalently, pairs  $((x,y),(z,w))\in G\times G$, where we ignore $w$ as it does not appear in the inequalities). 
The solution vector we are looking for is indexed by $P_XG$, and we denote it $(\varphi(x))_{x\in P_XG}$. In fact, formally, we should be using $(\varphi (x,y))_{(x,y)\in G}$, which seems to allow multi-valued $\varphi$. However, note that if $(x,y)$ and $(x,y')$ are both in $G$ then 
\[c(x,y) - c(x,y) \le  \varphi(x,y) - \varphi(x,y')  \]  and
\[c(x,y') - c(x,y') \le  \varphi(x,y') - \varphi(x,y)  \] 
which means 
\[   \varphi(x,y) = \varphi(x,y').  \] 
In other words, even if we do index the vector by $(x,y)\in G$ instead of $x\in P_XG$, the solution vector depends only on the first coordinate.

\subsection{Solutions for families of linear inequalities}

Our main theorem will follow from the next theorem regarding systems of linear inequalities.

\begin{manualtheorem}{1.2}\label{thm:linear-ineq-system-general}
	Let $\{\alpha_{i,j}\}_{i,j \in I}\in [-\infty, \infty)$, where $I$ is some arbitrary index set, and with $\alpha_{i,i} = 0$. The system of inequalities 
	\begin{equation}\label{eq:linear-sys} \alpha_{i,j} \le x_i - x_j, \quad i,j \in I \end{equation} 
	has a solution if (a) for any $i,j\in I$ there exists some constant $M(i,j)$ such that for any $m$ and any $i_2,\cdots,i_{m-1}$, letting $i = i_1$ and $j = i_m$ one has that  $\sum_{k=1}^{m-1} \alpha_{i_k, i_{k+1}}\le M(i,j)$, and \\ (b) either $I$ is at most countable, or, if $I$ is uncountable then for every {infinite subset} $J\subset I$ there exist some $j\in J, i\notin J$ with $\alpha_{j,i}>-\infty$.
\end{manualtheorem}

Instead of proving the theorem directly, we shall prove the following theorem, which at first glance might seem weaker.
\begin{thm}\label{thm:linear-ineq-system-weaker}
	Let $\{a_{i,j}\}_{i,j \in I}\in [-\infty, \infty)$, where $I$ is some arbitrary index set.  Assume that (a) for any $m\ge1$ and any $i_1,i_2,\cdots,i_m$ it holds that $a_{i_1, i_m} \ge \sum_{k=1}^{m-1} a_{i_k, i_{k+1}}$, and (b) either $I$ is at most countable, or, if $I$ is uncountable then for every {infinite subset} $J\subset I$ there exist some $j\in J, i\notin J$ with $\alpha_{j,i}>-\infty$.
\end{thm}

Clearly, Theorem \ref{thm:linear-ineq-system-general} implies Theorem \ref{thm:linear-ineq-system-weaker}. In fact, the reverse implication holds as well. We will show this by proving Theorem \ref{thm:linear-ineq-system-general} under the assumption of Theorem \ref{thm:linear-ineq-system-weaker}.

\begin{proof}[Proof that Theorem \ref{thm:linear-ineq-system-weaker} implies Theorem \ref{thm:linear-ineq-system-general}]
	
We will use Theorem \ref{thm:linear-ineq-system-weaker}. Assume that 
for any $i,j\in I$ there exists some $M({i, j})$ such that 
for any $m$ and any $\{i_k\}_{k=2}^{m-1}$, letting $i_1 = i$ and $i_m = j$   it holds that
\[ \sum_{k=1}^{m-1}\alpha_{i_k, i_{k+1}}\le M({i, j}).\] 
Define new constants $ {a}_{i,j}\in [-\infty, \infty)$ as follows: 
\[  {a}_{i,j} = \sup\{  \sum_{k=1}^{m-1}\alpha_{i_k, i_{k+1}}  : m\in {\mathbb N}, m\ge 2,\, i_2, \ldots, i_{m-1} \in I\}.\] 
By the above condition, the right hand side is bounded from above and so the supremum is not $+\infty$.

We first claim that the  system of equations ${a}_{i,j}  
\le   x_i - x_j,
$ satisfies the conditions of Theorem \ref{thm:linear-ineq-system-weaker}. Assume we are given $i_1, i_2,\cdots,i_{m-1}, i_m$, and we want to prove that $a_{i_1, i_m} \ge \sum_{k=1}^{m-1} a_{i_k, i_{k+1}}$. Fix $\eps>0$. For each $k\in [m]$ 
use the definition of $a_{i_{k}, i_{k+1}}$ to pick some $m_k$ and $i_2^{(k)}, \ldots, i_{m_k-1}^{(k)}$ such that, letting $i_1^{(k)} = i_k$ and $i_{m_k}^{(k)} = i_{k+1}$, we have  
\[ a_{i_k,i_{k+1}} \le  \sum_{l=1}^{m_k-1}\alpha_{i_l^{(k)}, i_{l+1}^{(k)}}  + \eps/m.\]

 We have thus identified some finite set of indices in $I$, the set 
 \[ J = \{i_k, i_{2}^{(k)}, \ldots, i_{m_k-1}^{(k)}: k\in [m-1] \}\cup\{i_m\},\] 
 which is naturally arranged as a path from $i_1$ to $i_m$.  Using again the 
 definition of $a_{i,j}$, the path thus defined participates in the supremum, and we have that 
 \[ a_{i_1,i_m} \ge 
  \sum_{k=1}^{m} ({a}_{i_k,i_{k+1}} - \eps /m) =  \big(\sum_{k=1}^{m} {a}_{i_k,i_{k+1}}\big)  - \eps.
 \]
As this holds for any $\eps$, we get the inequality in the condition (a) of Theorem \ref{thm:linear-ineq-system-weaker}.

Note that for any $J\subset I$ and $j\in J,\ i\in I\setminus J$, if $\alpha_{j,i}>-\infty$ then clearly $a_{j,i}\ge \alpha_{j,i}>-\infty$, demonstrating condition (b) of Theorem \ref{thm:linear-ineq-system-weaker}.
 
 Applying Theorem \ref{thm:linear-ineq-system-weaker}, we see that the system of inequalities
\begin{equation}\label{eq:system-we-want-2}
{a}_{i,j} %
\le   x_i - x_j,
\end{equation}  
admits a solution. Moreover, since $a_{i,j} \ge \alpha_{i,j}$ by definition, the resulting vector $x$ is also a solution of the original system of inequalities. 
\end{proof}

Having made the reduction from Theorem \ref{thm:linear-ineq-system-general} to Theorem \ref{thm:linear-ineq-system-weaker}, we proceed by proving the latter. The first case which we shall prove is when the index set $I$ is at most countable.

\subsection{The countable case}

\subsubsection{Graph theoretical component of the proof}
We shall make use of a decomposition of a nearly balanced weighted acyclic directed graph into weighted paths. A directed graph is called \emph{acyclic} if there are no directed cycles in the graph. We call a weighted graph \emph{nearly balanced} if we have some control over the difference between the in-coming and out-coming total weight in each vertex. We assume that all the weights are non-negative.

\begin{prop}\label{prop:combi-splitADGtopaths}
	Let $\,\Gamma = (V, %
	E, (w_e)_{e\in E})$ be a finite directed weighted acyclic graph. Assume that it is almost balanced in the following sense: for some fixed vector $(\eps _v)_{v\in V}$ with non-negative entries, we have for every vertex $v\in V$ that
	\[ \sum_{(x,v)\in E}w_{(x,v)} - \sum_{(v,y)\in E}w_{(v,y)} \in [-\eps_v, \eps_v].\] 
	Then there exists a weighted decomposition of $\Gamma$ into paths $P_k = v_{1}^{(k)} \to v_2^{(k)} \to \cdots \to v_{m_k}^{(k)}$ with equal weights $\mu_k$ for each edge in $P_k$, such that 
	\begin{equation}\label{eq:sum-of-muk} w_e =\sum_{k: \, e\in P_k}\mu_k  \quad {\rm and}\quad \sum_{k} \mu_k < \frac12 \sum \eps_i. \end{equation} 
	Moreover, it holds individually for each $v\in V$ that 
	\begin{equation}\label{eq:sum-of-muk-on-a-vertex} \sum_{\{k: \, v = s_k \ \text{or} \  v = f_k\}} \mu_k < \eps_v.\end{equation}	
	Here $s_k = v_1^{(k)}$ denotes the starting vertex of the path $P_k$ and $f_k = v_{m_k}^{(k)}$ denotes its end point. 
	In fact, if   $\sum_{(x,v)\in E}w_{(x,v)} \le  \sum_{(v,y)\in E}w_{(v,y)}$ then $v\neq f_k$ for any $k$ and if $\sum_{(x,v)\in E}w_{(x,v)} \ge \sum_{(v,y)\in E}w_{(v,y)}$ then $v\neq s_k$ for any $k$. 
\end{prop}

\begin{rem}
	Note that \eqref{eq:sum-of-muk-on-a-vertex} implies \eqref{eq:sum-of-muk}, which can be seen by summing the inequalities in \eqref{eq:sum-of-muk-on-a-vertex}  over all $v\in V$. Then the right hand side becomes $\sum_{v } \eps_v $, while the left hand side is  
	\[  \sum_{v} \sum_{\{k: s_k = v~{\rm or} f_k = v \}} \mu_k, \]
	and since every path has precisely one starting point $s_k$ and one final point $f_k$, we get that each $\mu_k$ was summed twice, that is we get exactly $2\sum \mu_k$.  
\end{rem}

\begin{proof}[Proof of Proposition \ref{prop:combi-splitADGtopaths}]
	Note that if $|V| = 2$ then the claim is trivial as we can use just one path. We then have that $\eps_1 = \eps_2 = w_{(1,2)} = \mu_1$, where we used $1$ and $2$ as labels of the vertices, and assumed the only edge is $(1,2)$.
	
	We shall use induction on the number of edges. If $V$ is any set and $|E| = 1$ then the situation is exactly as in the first case we considered and there is nothing to prove.

	Assume we know the claim for $|E|<k$ and we are given a graph $\Gamma$ with $|E| = k$. Consider an edge $e_* = (x,y)$ with minimal weight $w_*$ and pick a maximal path $P$ which includes it (maximal in the sense that it cannot be extended to a longer path), say $s = v_{1}\to v_{2} \to \cdots \to v_{m}=f$. Maximality implies that there is no outgoing edge from its end vertex $f = v_{m}$, and no edge going into its start vertex $s = v_{1}$. In particular, the ``almost balanced'' restriction on $s$ reads  
	$\sum_{(s,y)\in E}w_{(s,y)} < \eps_{s}$ and  on $f$ reads  
	$\sum_{(x,f)\in E}w_{(x,f)} < \eps_{f}$. 
	Moreover, since $w_*$ was a minimal weight in the whole graph, it follows that $\eps_{s}> w_*$ and $\eps_{f} >w_*$.

	Define $\Gamma'$ to be a graph with the same vertices $V$ and edges $e\in E$, whose weights are defined as 
	\begin{align*}
	w'_e=
	\begin{cases}
	w_e-w_*  \ \ \ \ \ \ \ \ \ \ &\text{ if  $e\in P$,}\\
	w_e \ \ \ &\text{ otherwise.}
	\end{cases}
	\end{align*}
	Since $w_*$ was chosen as the minimal weight, we see that all new weights remain non-negative. Note that the edge $e_* = (x,y)$ now has weight zero and thus can be omitted. Therefore, the graph $\Gamma'$ is a directed weighted acyclic graph with at most $k-1$ edges. It satisfies the almost-balanced condition with a new vector $\eps'_v$ given by
	\[ \eps'_{s} = \eps_{s}- w_*, \,\,\, \eps_{f}' = \eps_{f}-w_*,\,\,\,{\rm and}\,\,\, \eps'_v = \eps_v \,\,\,{\rm for} \,\,\,v\in V\setminus \{s,f\}.\]
	Note that  $\sum_{v\in V}\eps_v' = \sum_{v\in V} \eps_v-2w_*$. 
	
	By the induction assumption, the new graph $\Gamma'$ has a weighted decomposition: that is, we can find paths $(P_k)_{k\in S}$ and weights $\mu_k$ such that 
	$\sum_{\{k:\, e \, \in P_k\}}\mu_k = w'_{e}$ and for every $v\in V$, we have
	\[\sum_{\{k:\, s_k = v \text{ or } f_k = v\}} \mu_k \le  \eps'_v. \]

	We add to the collection the path $P$ with a weight $w_*$ on each edge. 
	We claim that this constitutes the desired weighted decomposition of $\Gamma$. 
	
	Indeed, if we compute $\sum_{\{k:\, s_k = v \text{ or } f_k = v\}} \mu_k $ for a vertex which is neither $s$ nor $f$, i.e. not an end point of $P$, we get the same result as in $\Gamma'$ and hence it is at most $\eps_v' = \eps_v$. If we compute the sum for $v = s$ or $v = f$, we get the sum in $\Gamma'$ with added $w_*$, which is thus bounded by $\eps_v'+w_*=\eps_v$, as needed. 
	
	Finally, by construction, a vertex can be chosen as a starting vertex $s_k$ for some path $P_k$ only if, after equal weights were removed from its   inwards   and  outwards  pointing edges, there was no weight left in the inwards  pointing edges. In other words, only if   $\sum_{(x,v)\in E}w_{(x,v)}<\sum_{(v,y)\in E}w_{(v,y)}$. Similarly, a vertex can be chosen as $f_k$ for some path $P_k$ only if $\sum_{(x,v)\in E}w_{(x,v)}>\sum_{(v,y)\in E}w_{(v,y)}$, which completes the proof.
\end{proof}

\subsubsection{A result of Ky Fan}

For the proof, we use the following result of Ky Fan \cite{KyFan}:

\begin{thm}[Ky Fan]\label{thm:KyFan}
	Let $E$ be a locally convex, real Hausdorff vector space. Let $x_\nu \in  E$ be an indexed set of vectors with indices $\nu\in I$, and let $\alpha_\nu\in \RR$. Then the
	system of inequalities 
	$f(x_\nu)\ge \alpha_\nu$ for $\nu \in I$ has a solution $f\in E^*$ (that is, $f$ a continuous linear functional
	on $E$) is and only if the point $(0,1)\in E\times \RR$ does not belong to the closed convex cone $C \in E\times \RR$ spanned by the elements $\{(x_\nu, \alpha_{\nu}): \nu \in I\}$. 
\end{thm}

Since the proof is a direct application of Hahn Banach theorem, we include it here for completeness. 

\begin{proof}[Proof of  Theorem \ref{thm:KyFan}]
	Denote by $C$  the closed convex cone spanned by $(x_\nu, \alpha_\nu)$.  
	Assume the system has a solution $f$, then $f\times( -{\rm id_{\RR}})\ge 0$ on $C$. On the other hand $f\times( -{\rm id_{\RR}})(0,1) = -1$, so $(0,1) \not\in C$. 
	
	For the other direction, we use the Hahn-Banach theorem to separate $(0,1)$ from the closed convex cone $C$. That is, we find a linear functional $(h, a)\in E^* \times \RR$ such that for any $(x,\alpha) \in C$ one has $h(x) + a\alpha \ge b > h(0) +a  = a$.
	Since $C$ is a cone, $(0,0) \in C$ and so $0\ge b > a$. In particular we see that $h(x) + a \alpha \ge b$ can be written as $f(x) -\alpha \ge b/(-a) \ge 0$ for $f = h/(-a)\in E^*$. That is, we have found a solution $f$ satisfying $f(x_\nu) \ge \alpha_\nu$ for all $\nu$.    
\end{proof}

\subsubsection{Completing the  proof in the case of $I$ at most countable}

We are ready to prove Theorem \ref{thm:linear-ineq-system-weaker}, in the case where the family of inequalities is at most countable. 

\begin{proof}[Proof of Theorem \ref{thm:linear-ineq-system-weaker}, countable $I$]\label{pf:countable}
	We will use Theorem \ref{thm:KyFan} for 
	the space $\RR^I$ with the box topology. It is clearly Hausdorff and locally convex. 
	The convex cone is generated by the vectors $(e_i - e_j, a_{i,j})$, and so we must show that 
  the point $(0,1)$ has a neighborhood separated from this cone.  

	The neighborhood we pick is of the form $\prod_{i\in I} (-\eps_i, \eps_i) \times (1/2, \infty)$, where the sequence $\eps_i$ will be chosen in a way which depends only on $\{a_{i,j}\}_{i,j\in I}$.

To  define the neighborhood, using that $I$ is countable we fix an ordering $\le$ of it, and for every $i\in I$ we define
	\[ \eps_i = \frac{1}{5} 2^{-i} \frac{1}{\max \{  a_{k,j} : k\le i, j\le i\} +1}. \]
	Note that $\eps_i>0$ for every $i$.

	Then, towards a contradiction, assume that the set $\prod_{i\in I} (-\eps_i, \eps_i) \times (1/2, \infty)$, which is an open neighborhood of $(0,1)\in \R^I \times \R$ in the box topology, intersects with the closed convex cone generated by the vectors $(e_i - e_j, a_{i,j})$. 
	This means that there is some finite $J\subset I$, and a positive combination $\sum_{i,j\in J} \lambda_{i,j} (e_i - e_j, a_{i,j})$ which is inside this neighborhood. This condition amounts to 
	\begin{equation}\label{eq:cone-intersects-neighII}\sum_{j\in J} \lambda_{i,j} - \lambda_{j,i} \in (-\eps_i, \eps_i)\,\,\,\forall i\in J \quad {\rm and} \quad \sum_{i,j\in J} \lambda_{i,j}a_{i,j} \ge 1/2.   \end{equation} 
	
	Let $\Lambda$ be the matrix with entries $\lambda_{i,j}$. Note that subtracting any positive multiple of a permutation matrix from the matrix 
	$\Lambda$ has no effect on the sum on the left and only increases (by cyclic monotonicity) the sum on the right. Thus we may assume without loss of generality that the matrix $\Lambda$ is not larger (entry-wise) than any positive multiple of a permutation matrix.
	
	Consider the elements of $J$ as vertices of a weighted directed graph $\Gamma$, where we define the weights to be $w_{(i,j)} = \lambda_{i,j}$. The assumption that $\Lambda$ contains no positive multiple of a permutation matrix implies that the graph $\Gamma$ is acyclic. Moreover, the first condition in \eqref{eq:cone-intersects-neighII} means that the graph $\Gamma$ is almost balanced, up to the weights $(\eps_i)$, in the sense of Proposition 
	\ref{prop:combi-splitADGtopaths}. Using Proposition 
	\ref{prop:combi-splitADGtopaths}  we find paths $P_k = v_{i_1}^{(k)} \to v_{i_2}^{(k)} \to \cdots \to v_{i_{m_k}}^{(k)}$ and weights $\mu_k$  that cover the graph $\Gamma$ and satisfy for every vertex $v_i$ that
	\[  \sum_{\{k:\, s_k = v_i \text{ or } f_k = \,v_i\}} \mu_k \le \eps_i  \]
	(where we let $s_k = v_{i_1}^{(k)}$ and $f_k = v_{i_{m_k}}^{(k)}$ as before). 
	
	Denote by $A _k$ the adjacency matrix of the path $P_k$, so that $\Lambda = \sum \mu_k A_k$. Then
	\[ \sum_{i,j\in J} \lambda_{i,j}a_{i,j}  = \sum_{k} \mu_k \sum_{i,j\in J} (A_k)_{i,j} a_{i,j}.\] 
	Moreover, by the assumption on $a_{i,j}$ in the statement of the theorem we are proving, and since $(A_k)_{i,j}\in \{0,1\}$ and are indicating a path, we get that 
	\[\sum_{i,j\in J} (A_k)_{i,j} a_{i,j}\le a_{s_k, f_k}.\]
	Therefore,
	\[ \sum_{i,j\in J} \lambda_{i,j}a_{i,j} \le \sum_{k} \mu_k a_{s_k, f_k}=: S.\]
	Let us now decompose the sum $S$ according to the start and end vector of the path $P_k$ in the following way. Fix an ordering of the (finite number of) elements in $J$, and write
	\[  S= \sum_{l=1}^{|J|} \sum_{\{ k:\, \max (s_k, f_k) = l  \}} \mu_k a_{s_k, f_k}. \]
	Indeed, each path is summed exactly once, according to the quantity $l = \max (s_k, f_k)$.

	From the definition of $\eps_i$, it follows that  
	\[ a_{k,j} \le   \frac{1}{5\eps_{\max\{j,k\}} } 2^{-{\max\{j,k\}}}. \] 
	Using this estimate for the sum $S$, we get 
	\[  S\le  \sum_{l=1}^{|J|} \frac{1}{5\eps_{l} } 2^{-{l}} \sum_{\{ k: \, \max (s_k, f_k) = l  \}} \mu_k  . \]
	Recall that the set $\{k: \max (s_k, f_k) = l  \}$ is  the set of all paths which either start or terminate at $l$, but their other endpoint (start or end point) is smaller than $l$. From Proposition \ref{prop:combi-splitADGtopaths}   we know that 
	\[  \sum_{\{k: \, s_k = l~{\rm or}~f_k = l \}} \mu_k \le \eps_l,  \]
	so in particular
	\[  \sum_{\{k:\, \max(s_k ,f_k) = l \}} \mu_k \le \eps_l.  \]
	We thus may conclude that
	\[  S\le  \sum_{l=1}^{\infty} \frac{1}{5\eps_{l} } 2^{-{l}} \eps_l  = 1/5<1/2\]
	which is a contradiction to the assumption and the proof is complete.  
\end{proof}

\subsection{The uncountable case} 
In this section, we prove the remaining case of Theorem \ref{thm:linear-ineq-system-weaker}, namely when the index set $I$ is uncountable. To describe the extra condition  assumed in this case, it is convenient to use the notion of a ``black hole'' for a system of inequalities.

\begin{definition}\label{def:black-hole}
Consider a collection of numbers $\{ a_{i,j} \}_{i,j\in I} \in [-\infty, \infty )$, where $I$ is an index set. 
We say that the collection $\{a_{i,j}\}_{i,j\in I}$ has a \emph{black hole} in the index set $J_0\subset I$ 
if for all $j\in J_0$ and all $i\in I \setminus J_0$ we have that $a_{j,i}=-\infty$.
We say that it has a black hole of infinite 
cardinality if such $J_0$ exists and is of infinite 
cardinality. 
\end{definition}

The assumption $(b)$ in Theorem \ref{thm:linear-ineq-system-weaker} is that the system of inequalities does not have an infinite black hole.

\begin{rem}\label{rem:biggest-hole}
    By definition, if $J_\alpha\subset I$ are black holes for the system $\{ a_{i,j} \}_{i,j\in I} \in [-\infty, \infty )$ for any $\alpha \in A$ then so is $J=\cup_{\alpha \in A} J_\alpha$. This means that one can take a maximal black hole $J\subseteq I$ by taking the union over all black holes, and this $J$ includes, as a subset, any black hole of any cardinality. 
\end{rem}

\begin{proof}[Proof of Theorem \ref{thm:linear-ineq-system-weaker}, uncountable case]

We start by letting $J_1$ be the union of all black holes in $I$. By Remark \ref{rem:biggest-hole}, the set $J_1$ is a black hole and by the \emph{added} assumption $J_1$ is finite. We take any countably infinite set $J_2\subseteq I$ which includes it. 
Then the system of inequalities indexed by $J_2$ has a solution due to Theorem \ref{thm:linear-ineq-system-weaker}. Denote this solution by $f_2$. 

	We shall now use Zorn's Lemma.  
	Consider the partially ordered set of pairs $(J, f_J)$ where $J_2 \subset J \subset I$ and $f_J : J\to \RR$ satisfies $f_J|_{J_2} = f_2$, and such that for any $i,j\in J$ we have $f_J(i)-f_J(j) \ge a_{i,j}$. 
	We know the set is non-empty because it contains the pair $(J_2, f_2)$. 
	The partial order we consider is $(J, f_J) \le (K, f_K)$ if $J\subset K$ and $f_K|_J = f_J$.
	
	First, let us notice  that every chain has an upper bound. Assume $(J_\alpha, f_{{J_\alpha}})_{\alpha \in A}$ is a chain (namely any two elements are comparable). Consider $J = \cup_\alpha J_\alpha$ and $f_J = \cup_\alpha f_{J_\alpha}$. This function is well defined because of the chain properties (at a point $i\in J$ it is defined as $f_{J_\alpha}(i)$ for any $\alpha$ with $i\in J_{\alpha}$). 
	The pair $(J, f|_J)$ is in our set because if $i,j\in J$ then for some $\alpha$ we have $i,j \in J_\alpha$, so $f|_{J_\alpha}$ satisfies the inequality on $f_J(i)-f_J(j) \ge a_{i,j}$ and so does $f_J$. 
	Finally, $(J, f_J)$ is clearly an upper bound for the chain.  So, we have shown that every chain has an upper bound, and we may use Zorn's lemma to find a maximal element. Denote the maximal element by $(J_0, f_{J_0})$.

	Assume towards a contradiction that $J_0 \neq I$. Note that the non-empty set $I\setminus J_0$ has no black holes since we assumed that $J_2 \subset J_0$ contains all the black holes in $I$. Therefore, there is some $i_0\in I\setminus J_0$ and some $j_{i_0}\in J_0$ with $a_{i_0,j_{i_0}}\neq -\infty$. 
    
     If we are able to extend $f_{J_0}$ to be defined on $\{i_0\}$ in such a way that all inequalities with indices of the form $(i_0, j)$ and $(j, i_0)$ with  $j\in J_0$  still hold, we will contradict maximality and complete the proof.
	
First, recall that under our assumptions $a_{k,j} \ge    a_{k, i_0} +a_{i_0,  j}$. Moreover, since $f_{J_0}$ already satisfies the inequality  
	$a_{k,j}\le    f_{J_0}(k) -f_{J_0}(j)$ for all $j,k \in J_0$, we get  that 
	\[ a_{k, i_0} +a_{i_0,  j} \le  a_{k,j}\le    f_{J_0}(k) -f_{J_0}(j)\] 
	holds for all $j,k \in J_0$. In particular, this gives us
	\[  a_{i_0,  j} + f_{J_0}(j) 
	\le    f_{J_0}(k) - a_{k, i_0}\]
	for any $j,k\in J_0$.
	This means that $f_{J_0}$ must satisfy that 
	\begin{equation}\label{eq:whatZORN} \sup_{j\in J_0}  \left(a_{i_0,  j} + f_{J_0}(j)\right) 
	\le \inf_{j\in J_0} \left(f_{J_0}(j) - a_{j, i_0}\right).\end{equation}
	
	 In particular, since we chose $i_0$ so that there exists $j_{i_0}\in J_0$ with $a_{i_0,j_{i_0}}\neq -\infty$ we know that the supremum is not $-\infty$, and therefore, the infimum is not $-\infty$. We will now show that  $\inf_{j\in J_0} \left(f_{J_0}(j) - a_{j, i_0}\right)$ is not $+\infty$, from which we will conclude that both the infimum and supremum are finite. 
    
To this end, we will show that for all $i\in I\setminus J_0$ there is some $j\in J_0$ such that $a_{j,i}\neq -\infty$.    Let $J_3 \subseteq I\setminus J_0$
denote all those $i\in I\setminus J_0$ for which there is some $j_i\in J_0$ with $a_{j_i,i}\neq -\infty$. We claim that $J_3 = I\setminus J_0$. Towards a contradiction,  assume that $I\setminus (J_0 \cup J_3)\neq \emptyset$. Then, as $J_0\cup J_3$ is not a black hole (since it has infinite cardinality), there is some $k\in J_0\cup J_3$ such that $a_{k,l}\neq -\infty$ for some $l \in I\setminus(J_0\cup J_3)$. The fact that $l\not\in J_3$ means $k\in J_3$ (and not in $J_0$). However, since $k\in J_3$ there is some $j_k\in J_0$ with $a_{j_k,k}\neq -\infty$. Together with our assumption that $a_{i_1,i_3} \ge    a_{i_1, i_2} +a_{i_2,  i_3}$ for any indexes $i_1, i_2, i_3 \in I$, this means 
\[ 
a_{j_k,l} \ge a_{j_k,k} + a_{k,l}>-\infty 
\]
 in contradiction to the fact that $l\not\in J_3$. Hence, as claimed, for all $i\in I\setminus J_0$ there is some $j\in J_0$ such that $a_{j,i}\neq -\infty$. In particular, this is true for $i=i_0$ which we chose before.

 We conclude that both sides of the inequality \eqref{eq:whatZORN} are finite, and hence we may take $f(i_0)\in \RR$ such that 
	\[ \sup_{j\in J_0}  \left(a_{i_0,  j} + f_{J_0}(j)\right) 
	\le  f(i_0) \le \inf_{j\in J_0} \left(f_{J_0}(j) - a_{j, i_0}\right).\]
	
This means that we may extend the function $f_{J_0}$ to $i_0$, which is a contradiction to the maximality, and we conclude that $J_0 = I$. This finished the proof as we have found a solution to the full system of inequalities.  
\end{proof}

\subsection{Summary}\label{subsec:summanry}

\begin{proof}[Proof of Theorem \ref{thm:Non-tradRR}]
	Assume that $G\subset X\times Y$ is $c$-path-bounded and has no infinite black holes (see Definition \ref{def:black-hole-cost}). By Theorem \ref{thm:potential-means-inequalities} we needed to show that the family of inequalities 
	\[ c(x,y) - c(z,y) \le \varphi(x) - \varphi(z), \]
	where $(x,y), (z,w)\in G$, has a solution. By Theorem \ref{thm:linear-ineq-system-general}, as $G$ is  $c$-path-bounded, condition (a) holds for $a_{x,z} = c(x,y) - c(z,y)$. Finally, since $G$ is either countable (in which case condition (b) holds) or does not have infinite black holes, in which case for any infinite subset $G_0 \subset G$, there is some $(x,y)\in G_0, (w,z)\in  G\setminus G_0$ with $c(z,y)<\infty$, i.e. with $a_{x,z} = c(x,y) - c(z,y)>-\infty$, so that condition (b) holds. Therefore, a solution to the above family exists.
\end{proof}

\subsection{Rockafellar-Rochet-R\"uschendorf theorem}

As a corollary we have a new and simple proof for the real-valued Rockafellar type theorem for traditional costs.  

\begin{cor}[Rockafellar-Rochet-R\"uschendorf]\label{cor:RRrvCOST}
	Let $c: X \times Y \to \RR$ be a traditional (i.e. finitely valued) cost function. Assume we are given a set $G \subset X\times Y$ which is $c$-cyclically monotone. Then there exists a $c$-class function $\varphi: X\to [-\infty, \infty]$ such that $G \subset \partial^c\varphi$. 	
\end{cor}

 \begin{proof}
 	In order to apply Theorem \ref{thm:Non-tradRR} we proceed to show that $G$ is $c$-path-bounded, as $G$ clearly has no black holes. Let
 	 $(x,y),(z,w)\in G$, and  let $M :=M((x,y), (z,w)) = c(x,w)-c(w,z)$. To show this indeed satisfies the condition, namely that there is an upper bound on the total cost of any path, we use the $c$-cyclic monotonicity. Let $m\ge 2$ and let $(x_2, y_2), \ldots, (x_{m-1}, y_{m-1}) \in G$. Denote $(x, y) = (x_1, y_1) $ and $(z,w) = (x_m,y_m)$.  
 	The condition of $c$-cyclic monotonicity implies that
 	\[  \sum_{i=1}^{m-1}\left( c(x_i, y_i) - c(x_{i+1}, y_i) \right) + c(x_m, y_m) - c(x_1, y_m) \le 0.   \]
 	In particular, 
 	\[ 	  \sum_{i=1}^{m-1}\left( c(x_i, y_i) - c(x_{i+1}, y_i) \right) \le c(x,w)-c(z,w). \] 
 \end{proof}

 It is important to note that we relied heavily on the fact that $c(x,w) <\infty$, otherwise this upper bound might be infinite, and therefore meaningless.

 \section{Special cases}\label{sec:sc}
  
  We now show that under certain assumptions, $c$-path-boundedness is implied by $c$-cyclic monotonicity, even for non-traditional costs. This is motivated by the fact  
   that usually, when considering \emph{optimal} transport plans, they are concentrated on sets which are automatically $c$-cyclically monotone (see \cite{pratelli,optimal-and-better}).  So, it is useful to indicate cases in which the 
$c$-cyclical monotonicity property implies that the set is also $c$-path-bounded. 
  
In this section we collect several such results. %
Recall the equivalence relation $\sim$ described in Subsection \ref{subsec:known-results}, and note that if the points $ (x_s,y_s)$ and $(x_f,y_f)$ are in the same equivalence class, then  there is a constant  $M$ as required in Theorem \ref{thm:Non-tradRR}. Indeed, fix an arbitrary path  from  $(x_f, y_f)$ to
$(x_s, y_s)$, say $(z_1, w_1), \ldots, (z_k, w_k)$, the edges of which are graph edges, and let \[ M = -\left[ c(x_f, y_f) - c(z_1, y_f) + \sum_{j=1}^{k-1} \big(c(z_i, w_i) - c(z_{i+1}, w_i)\big) + c(z_k, w_k) - c(x_s, w_k)\right] <\infty.\] 

Any path from $(x_s, y_s)$ to $(x_f, y_f)$ (paths not on the graph have total cost $-\infty$) is completed to a cycle using the above path, and using $c$-cyclic monotonicity we have for any $m$ and any $(x_i, y_i)_{i=2}^m\subset G$
\begin{eqnarray*}  c(x_s, y_s) - c(x_2, y_s) + \sum_{i=2}^{m-1} \big(c(x_i, y_i) - c(x_{i+1}, y_i)\big) + c(x_m, y_m) - c(x_f, y_m)  \le M .\end{eqnarray*} 

Moreover, if all the points are in one equivalence class then, clearly, there are no black holes. Summarizing, we gave a proof of Proposition \ref{prop:1equivClassIMPLIEScpathbdd}, since by Theorem \ref{thm:Non-tradRR} a $c$-path-bounded set with no black holes admits a $c$-potential.

  \begin{cor}\label{cor:ContConnect}
 	Let $c: X\times Y \to (-\infty,  \infty]$   be a continuous cost function on separable metric spaces $X,\,Y$  and let $G\subset X \times Y$ be $c$-cyclically monotone and  {path} connected.  Let $T>0$  and assume $G$ satisfies that   $c(x,y) \le T$ for all $(x,y)\in G$. Then $G$ is $c$-path-bounded. 
 \end{cor}

  \begin{proof}
  It is enough to show that any two points $(x,y),(z,w)\in G$ satisfy $(x,y)\sim (z,w)$. 
By  {path} connectivity we may find a continuous $\gamma:[0,1]\to G$ with $\gamma(0) = (x,y)$ and $\gamma(1) = (z,w)$. The compact set $\gamma([0,1])$ is of positive distance to the closed (by continuity of $c$) set $\{ (x,y): c(x,y) = \infty\}$. In particular, there exists some $\delta>0$ such that if $|(\eta,\xi) - \gamma(t)|<\delta$ then, denoting $\gamma(t)  = (x_t, y_t)$, we have $c(x_t,\xi)<\infty$. 
Since $\gamma$ is uniformly continuous, we may find $m$ and $t_0 = 0 <t_1<\cdots <t_m = 1$ such that   
$|\gamma(t_j) - \gamma(t_{j-1}) |<\delta$ for $j = 1, \ldots, m$. Denote $\gamma(t_j) = (x_j, y_j)$. 
Since $|(x_{j+1}, y_j) - (x_j, y_j)| \le  |(x_{j+1}, y_{j+1}) - (x_j, y_j)|<\delta $, we get that   $c(x_{j+1}, y_j)<\infty$, which means that the path $(\gamma(t_j))_{j=1}^m$ connects the points $(x,y)$ and $(z,w)$ in the graph. By symmetry, we get that any two points are connected and there is only one equivalence class for the relation $\sim$. Applying Proposition \ref{prop:1equivClassIMPLIEScpathbdd}, the proof is complete. 
   \end{proof}

As may be apparent from the proof of the corollary, the connectedness of $G$ is not elemental, and we may replace it with other assumptions, so long as these imply that there is only one equivalence class for $\sim$. Another useful variant, in which there may be more then one equivalence class, is the following.

\begin{prop}\label{prop:T-bdd}
	Let $c: X\times Y \to  (-\infty,\infty]$  be a continuous cost function on separable metric spaces $X,\,Y$   and  let $G\subset X \times Y$ be $c$-cyclically monotone and bounded.  Let $T>0$  and assume that for every $(x,y)\in G$ we have that $c(x,y)\le T$. Then $G$ is $c$-path-bounded. 
\end{prop}

\begin{proof} 
By the argument given
at the beginning of this section,
we only need 
 to address pairs which lie in different equivalence classes of the relation $\sim$, that is, show that for such pairs a bound on the total cost of a path between them exists.

We first observe that under the assumptions we have made, there are only finitely many equivalence classes for $\sim$. Indeed, let $S_T =\{ (x,y) : c(x,y) \le T\}$, then by continuity of $c$ on the compact set $\overline{G}\subset S_T$,   we can find some $\delta>0$ such that if $|(z,w)- (x,y)|<\delta$ and $(x,y), (z,w)\in G$  then $\max(c(x,w), c(z,y)) <T+1$. In particular, any two points in $G$ whose distance is less than $\delta$ belong to the same equivalence class. Using compactness of $G$ again, we may cover it with a finite number of $\delta$-balls, so there can be no more than a finite number of different equivalence classes. Denote the number of equivalence classes of $\sim$ by $k_0\in {\mathbb  N}$.

	Next, we claim that for each equivalence class ( denoted by $[v]$) there exists a bound $M = M([v])$ (depending only on the equivalence class) 	
 such that for any two points $(x_s, y_s)$ and $(x_f, y_f)$ in $[v]$, the total cost of any path between then is bounded by $M$.  
	 To this end define the function $F: [v]  \times [v] \rightarrow \RR$ to be the supremum over the total cost on any path from the first given point to the second one, namely let  $F((x_s,y_s),(x_f,y_f))$ be given by  
\[	 \sup \left\lbrace c(x_s,y_s)- c(x_1,y_s)+ \sum_{k=1}^{m-1} \left( c(x_k,y_k)-c(x_{k+1},y_k) \right) + c(x_m,y_m)-c(x_f,y_m) \right\rbrace \] where the supremum runs over all $m  {\in \N}$ and any $(x_i,y_i)_{i=1}^m \in G$.
We have already shown (using $c$-cyclic monotonicity  {and the definition of the relation $\sim$}) that $F$ is finite. To see that it is bounded, it suffices to show that $F$ is uniformly continuous (since the domain is bounded as well).

	 Let  $(x_s, y_s),(x_f, y_f)\in [v]$. Given some path $(x_i, y_i)_{i=1}^m$ joining $(x_s',y_s')$ and $(x_f',y_f')$, we may add to it the two points $(x_s,y_s)$ and $(x_f,y_f)$ as the first and last points, getting a new path between $(x_s,y_s)$ and $(x_f,y_f)$. We thus see that 
	   \begin{eqnarray*}
	   F ((x_s',y_s'),(x_f',y_f')) + c(x_s, y_s) - c(x_s', y_s) + c(x_f', y_f') - c(x_f, y_f')  \le    F((x_s,y_s),(x_f,y_f)) .  \end{eqnarray*} 
	   However, as   $c$   is continuous  on the compact set $\overline{G}\subset S_T$,   
	   it is uniformly continuous, and for any $\eps>0$ we may pick  $\delta = \delta(\eps)$  such that if 
	   $|(x_s, y_s) - (x_s', y_s')|<\delta$ and  $|(x_f, y_f) - (x_f', y_f')|<\delta$ then 
	    $\left|c(x_s, y_s) - c(x_s', y_s)\right|<\eps/2 $  and $\left| c(x_f', y_f') - c(x_f, y_f')\right|<\eps/2$ so that we get 
	  \begin{eqnarray*}
	 	F((x_s',y_s'),(x_f',y_f'))  - F((x_s ,y_s ),(x_f ,y_f ))\le \eps.  \end{eqnarray*} 
 	By symmetry  {of $F$ in its arguments}, we get that $F$ is indeed uniformly continuous on $[v]\times [v]$, and in particular bounded.  {Denote this bound by $M([v])$.}

	 Finally,  {by the definition of our equivalence relation,} any path joining two points in $G$  can be split into  {at most} $k_0$ paths, each one within one of the equivalence classes, and at most $ k_0-1$ extra steps, each joining  two equivalence classes.
	
 To bound each of the ``single steps'' joining two different equivalence classes, say joining $(z,w) \in [v_1]$ and $(z',w')\in [v_2]$, notice first that the cost $c$ is bounded from below on $\overline{P_XG\times P_YG}$ by continuity. Therefore,  
  \[ c(z,w) - c(z',w) \le T - \inf \{c(x,w): (x,y)\in G, (z,w)\in G\} = : M_2. \]
To sum up, denoting the $k_0$ equivalence classes by $([v_i])_{i=1}^{k_0}$, the total cost for any path in $G\subset S_T$ is  bounded from above by $\sum_{i=1}^{k_0}M([v_i]) +(k_0-1)M_2$. As a result we have a uniform bound for any path with any beginning and end point in $G$, that is, $G$ as a whole is $c$-path-bounded (in fact, uniformly, which is a much stronger statement).	
\end{proof}

Summing up, we have seen that under additional geometric or topological conditions, $c$-cyclic monotonicity is in fact enough, and does imply $c$-path-boundedness (and in particular, the existence of a $c$-potential). For example, in the results of \cite{gangbo-oliker}, where the assumptions on the cost are that it is continuous, and infinite only on the diagonal, it is easy to check that a compact $c$-cyclically monotone set satisfies the conditions in Proposition \ref{prop:T-bdd}.

\bibliographystyle{apalike}
\bibliography{ref}

\begin{bibdiv}
\begin{biblist}

\bib{ambrosio-gigli}{incollection}{
      author={Ambrosio, Luigi},
      author={Gigli, Nicola},
       title={A user’s guide to optimal transport},
        date={2013},
   booktitle={Modelling and optimisation of flows on networks},
   publisher={Springer},
       pages={1\ndash 155},
}

\bib{ambrosio-pratelli}{incollection}{
      author={Ambrosio, Luigi},
      author={Pratelli, Aldo},
       title={Existence and stability results in the ${L}^1$ theory of optimal
  transportation},
        date={2003},
   booktitle={Optimal transportation and applications},
   publisher={Springer},
       pages={123\ndash 160},
}

\bib{paper3}{unpublished}{
      author={Artstein-Avidan, Shiri},
      author={Barel, Hila},
      author={Rubinstein, Yanir},
      author={Sadovsky, Shay},
      author={Wyczesany, Katarzyna},
       title={Transportation induced by the polarity transform},
        note={In preparation},
}

\bib{artstein-florentin-milman}{article}{
      author={Artstein-Avidan, Shiri},
      author={Florentin, Dan},
      author={Milman, Vitali},
       title={Order isomorphisms in windows},
        date={2011},
     journal={Electronic Research Announcements},
      volume={18},
       pages={112\ndash 118},
}

\bib{hidden}{article}{
      author={Artstein-Avidan, Shiri},
      author={Milman, Vitali},
       title={Hidden structures in the class of convex functions and a new
  duality transform},
        date={2011},
     journal={Journal of the European Mathematical Society},
      volume={13},
      number={4},
       pages={975\ndash 1004},
}

\bib{shiri-yanir}{article}{
      author={Artstein-Avidan, Shiri},
      author={Rubinstein, Yanir},
       title={Differential analysis of polarity: {P}olar {H}amilton-{J}acobi,
  conservation laws, and {M}onge {A}mp{\`e}re equations},
        date={2017},
     journal={Journal d'Analyse Math{\'e}matique},
      volume={132},
      number={1},
       pages={133\ndash 156},
}

\bib{paper2}{article}{
      author={Artstein-Avidan, Shiri},
      author={Sadovsky, Shay},
      author={Wyczesany, Katarzyna},
       title={Optimal measure transportation with respect to non-traditional
  costs},
        note={arXiv:2104.04838},
}

\bib{hila}{thesis}{
      author={Barel, Hila},
       title={Optimal transportation problem for polar cost},
        type={Master's Thesis},
        date={2019},
}

\bib{optimal-and-better}{article}{
      author={Beiglb{\"o}ck, Mathias},
      author={Goldstern, Martin},
      author={Maresch, Gabriel},
      author={Schachermayer, Walter},
       title={Optimal and better transport plans},
        date={2009},
     journal={Journal of Functional Analysis},
      volume={256},
      number={6},
       pages={1907\ndash 1927},
}

\bib{general-duality-MK}{article}{
      author={Beiglb{\"o}ck, Mathias},
      author={L{\'e}onard, Christian},
      author={Schachermayer, Walter},
       title={A general duality theorem for the monge-kantorovich transport
  problem},
        date={2012},
     journal={Studia Mathematica},
      volume={209},
      number={2},
       pages={151\ndash 167},
}

\bib{bertrand-gauss-curvature}{article}{
      author={Bertrand, J{\'e}r{\^o}me},
       title={Prescription of {G}auss curvature using optimal mass transport},
        date={2016},
     journal={Geometriae Dedicata},
      volume={183},
      number={1},
       pages={81\ndash 99},
}

\bib{bertrand-pratelli-puel}{article}{
      author={Bertrand, J{\'e}r{\^o}me},
      author={Pratelli, Aldo},
      author={Puel, Marjolaine},
       title={Kantorovich potentials and continuity of total cost for
  relativistic cost functions},
        date={2018},
     journal={Journal de Math{\'e}matiques Pures et Appliqu{\'e}es},
      volume={110},
       pages={93\ndash 122},
}

\bib{bertrand-puel}{article}{
      author={Bertrand, J{\'e}r{\^o}me},
      author={Puel, Marjolaine},
       title={The optimal mass transport problem for relativistic costs},
        date={2013},
     journal={Calculus of Variations and Partial Differential Equations},
      volume={46},
      number={1-2},
       pages={353\ndash 374},
}

\bib{bianchini-caravenna}{article}{
      author={Bianchini, Stefano},
      author={Caravenna, Laura},
       title={On optimality of c-cyclically monotone transference plans},
        date={2010},
     journal={Comptes Rendus Mathematique},
      volume={348},
      number={11-12},
       pages={613\ndash 618},
}

\bib{bianchini-caravenna-long}{article}{
      author={Bianchini, Stefano},
      author={Caravenna, Laura},
       title={On optimality of c-cyclically monotone transference plans},
        date={2010},
     journal={Comptes Rendus Mathematique},
      volume={348},
      number={11-12},
       pages={613\ndash 618},
}

\bib{brenier}{article}{
      author={Brenier, Yann},
       title={Polar factorization and monotone rearrangement of vector-valued
  functions},
        date={1991},
     journal={Communications on pure and applied mathematics},
      volume={44},
      number={4},
       pages={375\ndash 417},
}

\bib{brenier-relativ-cost}{incollection}{
      author={Brenier, Yann},
       title={Extended {M}onge-{K}antorovich theory},
        date={2003},
   booktitle={Optimal transportation and applications},
   publisher={Springer},
       pages={91\ndash 121},
}

\bib{KyFan}{article}{
      author={Fan, Ky},
       title={On infinite systems of linear inequalities},
        date={1968},
     journal={Journal of Mathematical Analysis and Applications},
      volume={21},
      number={3},
       pages={475\ndash 478},
}

\bib{gangbo-oliker}{article}{
      author={Gangbo, Wilfrid},
      author={Oliker, Vladimir},
       title={Existence of optimal maps in the reflector-type problems},
        date={2007},
     journal={ESAIM: Control, Optimisation and Calculus of Variations},
      volume={13},
      number={1},
       pages={93\ndash 106},
}

\bib{jimenez-chloe-santambrigio}{article}{
      author={Jimenez, Chlo{\'e}},
      author={Santambrogio, Filippo},
       title={Optimal transportation for a quadratic cost with convex
  constraints and applications},
        date={2012},
     journal={Journal de math{\'e}matiques pures et appliqu{\'e}es},
      volume={98},
      number={1},
       pages={103\ndash 113},
}

\bib{kantorovich}{inproceedings}{
      author={Kantorovich, Leonid},
       title={On one effective method of solving certain classes of extremal
  problems},
        date={1940},
   booktitle={Dokl. akad. nauk. ussr},
      volume={28},
       pages={212\ndash 215},
}

\bib{kantorovich2}{article}{
      author={Kantorovich, Leonid},
       title={On a problem of {M}onge},
        date={2006},
     journal={J. Math. Sci.(NY)},
      volume={133},
       pages={1383},
}

\bib{oliker}{article}{
      author={Oliker, Vladimir},
       title={Embedding ${S}^n$ into ${R}^{n+ 1}$ with given integral {G}auss
  curvature and optimal mass transport on ${S}^n$},
        date={2007},
     journal={Advances in Mathematics},
      volume={213},
      number={2},
       pages={600\ndash 620},
}

\bib{pratelli}{article}{
      author={Pratelli, Aldo},
       title={On the sufficiency of c-cyclical monotonicity for optimality of
  transport plans},
        date={2008},
     journal={Mathematische Zeitschrift},
      volume={258},
      number={3},
       pages={677\ndash 690},
}

\bib{rochet}{article}{
      author={Rochet, Jean-Charles},
       title={A necessary and sufficient condition for rationalizability in a
  quasi-linear context},
        date={1987},
     journal={Journal of mathematical Economics},
      volume={16},
      number={2},
       pages={191\ndash 200},
}

\bib{rockafellar}{article}{
      author={Rockafellar, R.~Tyrrell},
       title={Characterization of the subdifferentials of convex functions},
        date={1966},
     journal={Pacific Journal of Mathematics},
      volume={17},
      number={3},
       pages={497\ndash 510},
}

\bib{rockafellar-book}{book}{
      author={Rockafellar, R.~Tyrrell},
       title={Convex {A}nalysis},
   publisher={Princeton university press},
        date={1970},
      number={28},
}

\bib{ruschendorf}{article}{
      author={R{\"u}schendorf, Ludger},
       title={On c-optimal random variables},
        date={1996},
     journal={Statistics \& probability letters},
      volume={27},
      number={3},
       pages={267\ndash 270},
}

\bib{knott-smith}{article}{
      author={Smith, Cyril},
      author={Knott, Martin},
       title={On {H}oeffding-{F}r{\'e}chet bounds and cyclic monotone
  relations},
        date={1992},
     journal={Journal of multivariate analysis},
      volume={40},
      number={2},
       pages={328\ndash 334},
}

\bib{villani-book}{book}{
      author={Villani, C{\'e}dric},
       title={Optimal transport: old and new},
   publisher={Springer Science \& Business Media},
        date={2008},
      volume={338},
}

\end{biblist}
\end{bibdiv}

\smallskip \noindent
School of Mathematical Sciences, Tel Aviv University, Tel Aviv 69978, Israel  
\smallskip \noindent

{\it e-mail}: shiri@tauex.tau.ac.il\\
{\it e-mail}: shaysadovsky@mail.tau.ac.il\\
{\it e-mail}: kasiawycz@outlook.com\\

\end{document}